\documentclass[preprint]{elsarticle}

\usepackage{hyperref}

\journal{Discrete Mathematics}

\usepackage{amsmath,amsthm,amssymb}
\usepackage[utf8]{inputenc}
\usepackage{enumerate}
\usepackage{mathtools}

 \newcommand*{\double}[2][.1ex]{%
  \mathrel{\vcenter{\offinterlineskip%
  \hbox{$#2$}\vskip#1\hbox{$#2$}}}}

\newcommand{\opp}{{\mathrm{opp}}}

\newcommand{\arcC}{\mathrm{C}_A}
\newcommand{\arcW}{\mathrm{W}}

\newcommand{\Ical}{{\mathcal{I}}}

\newcommand{\Wcal}{{\mathcal{W}}}

\newcommand{\Fbb}{{\mathbb{F}}}

\newcommand{\Nbb}{{\mathbb{N}}}

\newcommand{\Qbb}{{\mathbb{Q}}}

\newcommand{\Pbf}{{\mathbf{P}}}

\newcommand{\restrict}{|}
\newcommand{\contract}{.}

\newcommand{\BS}{\backslash}

\newcommand{\rk}{\mathrm{rk}}

\newcommand{\N}{\mathbb{N}}

\newcommand{\disunion}{\mathbin{\dot{\cup}}}

\renewcommand{\phi}{\varphi}

\newcommand{\R}{\mathbb{R}}

\newcommand{\maparrow}{\longrightarrow}

\newcommand{\BSET}[1]{{\BS\left\{ #1 \right\}}}
\newcommand{\SET}[1]{{\left\{ #1 \right\}}}

\newcommand{\COMMENT}[1]{}
\newcommand{\XXXCUTXXX}[1]{}

\newcommand{\ROMANENUM}{\renewcommand{\theenumi}{(\roman{enumi})}\renewcommand{\labelenumi}{\theenumi}}

\newcommand{\routesto}{\double{\rightarrow}}

\newcommand{\deftext}[2][]{\emph{#2}}
\newcommand{\deftextX}[2][]{\emph{#2}}

\newcommand{\PRFR}[1]{\ignorespaces}
\newcommand{\bm}{\ignorespaces}

\newtheorem{theorem}{Theorem}[section]

\newtheorem{conjecture}[theorem]{Conjecture}
\newtheorem{lemma}[theorem]{Lemma}

\theoremstyle{definition}
\newtheorem{definition}[theorem]{Definition}
\newtheorem{remark}[theorem]{Remark}

\begin{document}

\begin{frontmatter}


\title{Duality Respecting Representations and Compatible Complexity Measures for Gammoids}

\author{Immanuel Albrecht}
\ead{mail@immanuel-albrecht.de}
\address{FernUniversit\"at in Hagen,\\
Department of Mathematics and Computer Science,\\
Discrete Mathematics and Optimization\\
D-58084~Hagen, Germany}


\begin{abstract}
We show that every gammoid has special digraph representations, such that a representation of the dual of the gammoid
may be easily obtained by reversing all arcs. In an informal sense, the duality notion of a poset applied to
the digraph of a special representation of a gammoid commutes with the operation of forming the dual of that gammoid.
We use these special representations in order to define a
complexity measure for gammoids, such that the classes of gammoids with bounded complexity are closed under duality,
minors, and direct sums.
\end{abstract}


\end{frontmatter}

A well-known result due to J.H.~Mason is that the class of gammoids is closed under duality, minors, and direct sums \cite{M72}.
Furthermore, it has been shown by D.~Mayhew that every gammoid is also a minor of an excluded minor for the class of gammoids \cite{Ma16},
which indicates that handling the class of all gammoids may get very involved.
In this work, we introduce a notion of complexity for gammoids which may be used to define subclasses of gammoids 
with bounded complexity, that
still have the desirable property of being closed under duality, minors, and direct sums; yet their representations have a more limited
number of arcs than the general class of gammoids.


\section{Preliminaries}

In this work, we consider \emph{matroids} to be pairs $M=(E,\Ical)$ where $E$ is a finite set 
and $\Ical$ is a system of
independent subsets of $E$ subject to the usual axioms (\cite{Ox11}, Sec.~1.1).
If $M=(E,\Ical)$ is a matroid and $X\subseteq E$, then the restriction of $M$ to $X$
shall be denoted by $M\restrict X$ (\cite{Ox11}, Sec.~1.3), and the contraction of $M$ to $X$ shall be denoted by $M\contract X$
(\cite{Ox11}, Sec.~3.1).
Furthermore, the notion of a \emph{digraph} shall be synonymous with what is described more 
precisely as \emph{finite simple directed graph} that may have some loops, i.e. a digraph is 
a pair $D=(V,A)$ where $V$ is a finite
set and $A\subseteq V\times V$. 
Every digraph $D=(V,A)$ has a unique \emph{opposite digraph} $D^\opp = (V,A^\opp)$ where
$(u,v)\in A^\opp$ if and only if $(v,u)\in A$.
All standard notions related to digraphs in this work are in
accordance with the definitions found in \cite{BJG09}. A \emph{path} in $D=(V,A)$ is
a non-empty and non-repeating sequence $p = \left(p_1 p_2 \ldots p_n\right)$ of vertices $p_i\in V$ such that
for each $1 \leq i < n$, $(p_i,p_{i+1})\in A$. By convention, we shall denote $p_n$ by $p_{-1}$.
Furthermore, the set of vertices traversed by a path $p$ shall be denoted by $\left| p \right| = \SET{p_1,p_2,\ldots,p_n}$
and the set of all paths in $D$ shall be denoted by $\Pbf(D)$.

\begin{definition}\PRFR{Jan 22nd}
	Let $D = (V,A)$ be a digraph, and $X,Y\subseteq V$. A \deftext{routing} from $X$ to $Y$ in $D$ is a family of paths $R\subseteq \Pbf(D)$ such that
	\begin{enumerate}\ROMANENUM
		\item for each $x\in X$ there is some $p\in R$ with $p_{1}=x$,
		\item for all $p\in R$ the end vertex $p_{-1}\in Y$, and
		\item for all $p,q\in R$, either $p=q$ or $\left|p\right|\cap \left|q\right| = \emptyset$.
	\end{enumerate}
	We shall write $R\colon X\routesto Y$ in $D$ as a shorthand for ``$R$ is a routing from $X$ to $Y$
    in $D$'', and if no confusion is possible, \label{n:routing}
    we just write $X\routesto Y$ instead of $R$ and $R\colon X\routesto Y$.
    A routing $R$ is called \deftext{linking} from $X$ to $Y$, if it is a routing onto $Y$, i.e. whenever $Y = \SET{p_{-1}\mid p\in R}$.
\end{definition}

\begin{definition}\label{def:gammoid}\PRFR{Jan 22nd}
    Let $D = (V,A)$ be a digraph, $E\subseteq V$,
    and $T\subseteq V$. 
    The \deftext[gammoid represented by DTE@gammoid represented by $(D,T,E)$]{gammoid represented by $\bm{(D,T,E)}$} is defined to be the matroid $\Gamma(D,T,E)=(E,\Ical)$\label{n:GTDE}
     where
    \[ \Ical = \SET{X\subseteq E \mid \text{there is a routing } X\routesto T \text{ in D}}. \]
    The elements of $T$ are usually called \deftextX{sinks} in this context, although they are not required to be actual sinks of the digraph $D$. To avoid confusion, 
    we shall call the elements of $T$ \deftext{targets} in this work. A matroid $M'=(E',\Ical')$ is called \deftextX{gammoid}, if there is a digraph $D'=(V',A')$ and a set $T'\subseteq V'$ such that $M' = \Gamma(D',T',E')$.
\end{definition}

\begin{theorem}[\cite{M72}, Corollary~4.1.2]\label{thm:baseRepresentation}
    Let $M=(E,\Ical)$ be a gammoid and $B\subseteq E$ a base of $M$. Then there is 
    a digraph $D=(V,A)$ such that $M = \Gamma(D,B,E)$.
\end{theorem}
For a proof, see J.H.~Mason's seminal paper \emph{On a Class of Matroids Arising From Paths in Graphs} \cite{M72}.
Here, we are content with pointing out that the proof is constructive and involves a sequence of
the following kind of construction:

\begin{definition}\label{def:rsswap}\PRFR{Jan 22nd}
  Let $D=(V,A)$ be a digraph, and let $(r,s)\in A$ be an arc of $D$. 
  The \emph{swap of $(r,s)$ in $D$} 
	shall be the digraph 
  $D_{(r,s)} = (V,A_{(r,s)})$ where
  \[ A_{(r,s)} = \SET{(u,v)\in A ~\middle|~ u\not= r} \cup \SET{(s,v)~\middle|~v\not= s \text{~and~} (r,v)\in A} \cup \SET{(s,r)}.\]
\end{definition}

First, observe that the number of arcs in the swap of $(r,s)$ in $D$ is bounded by the number of arcs in the original digraph $D$.

The proof of Theorem~\ref{thm:baseRepresentation} in \cite{M72} is based on the observation that the gammoid
$\Gamma(D,T,E)$ is the same as the gammoid
$\Gamma(D_{(r,s)},T_{(r,s)},E)$
where $D=(V,A)$, $T_{(r,s)} = \left(T\backslash\SET{s}\right)\cup\SET{r}$, and $(r,s)\in A$ such that $s\in T$ and $r\notin T$.

\begin{remark}\label{rem:swapsequence}
For every representation $(D,T,E)$ and every routing $R \colon B \routesto T$, there is a straight-forward sequence of swaps  
that yields a representation of the desired form: assume that $R = \SET{r^{(1)},\ldots,r^{(n)}}$ thus fixing a particular order of the paths in $R$.
Each path in the routing corresponds to a contiguous subsequence of swaps which appear in the above order.
The path $r^{(i)} = \left(r_1^{(i)}r_2^{(i)}\ldots r_{n_i}^{(i)}\right)$ corresponds to a block of $n_i - 1$ swaps, indexed by $k = 1,2,\ldots, n_i-1$.
The $k$th swap in the subsequence corresponding to $r^{(i)}$ is the swap of $\left(r_{n_i -k}^{(i)},r_{n_i-k+1}^{(i)}\right)$ in the digraph
$D^{(k+K_i-1)}$, which yields the digraph 
\[ D^{(k+K_i)} = \left(D^{(k+K_i-1)}\right)_{\left(r_{n_i -k}^{(i)},r_{n_i-k+1}^{(i)}\right)}
\]
where $K_i = \sum_{j=1}^{j<i} (n_j-1)$ and $D^{(0)} = D$.

In other words, given a basis $B$ of $\Gamma(D,T,E)$, we may swap all arcs of a given routing $B\routesto T$ in $D$ in the reverse order of traversal in their respective paths,
which gives us a digraph $D'$ such that $\Gamma(D,T,E) = \Gamma(D',B,E)$. Furthermore, all $b\in B\backslash T$ will be sinks in $D'$, and for the sake of reduction of complexity, 
we remove all arcs that start in $b\in B\cap T$ from $D'$, too; this does not change the represented gammoid.
\end{remark}

%
%
%
%
\section{Special Representations}

\begin{definition}\label{def:dualityRespectingRepr}\PRFR{Jan 22nd}
	Let $(D,T,E)$ be a representation of a gammoid. We say that $(D,T,E)$ is a \deftext{duality respecting representation},
	if 
	\[ \Gamma(D^\opp,E\BS T,E) = \left( \Gamma(D,T,E) \right)^\ast  \]
    where $\left( \Gamma(D,T,E) \right)^\ast$ denotes the dual matroid of $\Gamma(D,T,E)$.
\end{definition}

\begin{lemma}\label{lem:dualityrespectingrepresentation}\PRFR{Jan 22nd}
  Let $(D,T,E)$ be a representation of a gammoid with $T\subseteq E$,
  such that every $e\in E\BS T$ is a source of $D$, and every $t\in T$ is a sink of $D$.
    Then $(D,T,E)$ is a duality respecting representation.
\end{lemma}
\begin{proof}\PRFR{Jan 22nd}
  We have to show that the bases of $N = \Gamma(D^\opp, E\BS T, E)$ are precisely the complements of the
  bases of $M = \Gamma(D,T,E)$ (\cite{Ox11}, Thm.~2.1.1).
  Let $B\subseteq E$ be a base of $M$, then there is a linking $L\colon B\routesto T$ in $D$, and since $T$ consists of sinks,
  we have that the single vertex paths $\SET{(x)\in\Pbf(D)~\middle|~x\in T\cap B} \subseteq L$.
  Further, let $L^\opp = \SET{ \left(p_n p_{n-1} \ldots p_1 \right) \mid  \left(p_1 p_2 \ldots p_n\right) \in L}$.
  Then $L^\opp$ is a linking
  from $T$ to $B$ in $D^\opp$ which routes $T\BS B$ to $B\BS T$. The special property of $D$, that $E\BS T$ consists of sources and that $T$ consists of sinks,
  implies, that for all $p\in L$, we have $\left| p \right|\cap E = \SET{p_1,p_{-1}}$.
  Observe that thus
 $$R = \SET{p\in L^\opp \mid p_{1}\in T \BS B} \cup \SET{(x)\in \Pbf(D^\opp)\mid x\in E\BS \left( T\cup B \right)}$$
  is a linking from $E\BS B=\left( T\disunion \left( E\BS T \right) \right)\BS B$ onto $E\BS T$ in $D^\opp$, thus $E\BS B$ is a base of $N$. An analog argument yields that for every base $B'$ of $N$, the set
  $E\BS B'$ is a base of $M$. Therefore $\Gamma(D^\opp, E\BS T, E) = \left( \Gamma(D,T,E) \right)^\ast$.
\end{proof}

\begin{definition}\label{def:standardRepresentation}\PRFR{Mar 27th}
	Let $M$ be a gammoid and $(D,T,E)$ with $D=(V,A)$ be a representation of $M$.
	Then $(D,T,E)$ is a \deftext[standard representation of a gammoid]{standard representation of $\bm M$},
  if 
  \begin{enumerate}\ROMANENUM
   \item $T\subseteq E$, 
   \item every $t\in T$ is a sink in $D$, and
   \item every $e\in E\BS T$ is a source in $D$.
  \end{enumerate}
\end{definition}

The name \emph{standard representation} is justified, since the real matrix representation $A\in \R^{T\times E}$ of $M$
 obtained from $D$ through the Lindström Lemma \cite{Li73,Ardila2007}
 is a \emph{standard matrix representation} of $\Gamma(D,T,E)$ up to possibly rearranging the columns (\cite{We71}, p.137).
 
 \begin{theorem}\label{thm:two-four}\label{lem:sourcesinkrepresentation}\PRFR{Jan 22nd}
	Let $M=(E,\Ical)$ be a gammoid, and $B\subseteq E$ a base of $M$.
	There is a digraph $D=(V,A)$ such that $(D,B,E)$ is a standard representation of $M$.
\end{theorem}
\begin{proof}
  Let $D_0=(V_0,A_0)$ be a digraph such that $\Gamma(D_0,B,E) = M$ (Theorem~\ref{thm:baseRepresentation}).
  Furthermore, let $V$ be a set with $E\subseteq V$ such that there is an injective map $ '\colon V_0 \maparrow V\BS E,\,v\mapsto v'$.
  Without loss of generality we may assume that $V = E\disunion V_0'$.
  We define the digraph $D=(V,A)$ such that $$A = \SET{(u',v') ~\middle|~ (u,v)\in A_0} \cup \SET{(b',b) ~\middle|~ b\in B} \cup \SET{(e,e')~\middle|~ e\in E\BS B}.$$ 
  For every $X\subseteq E$, we obtain that by construction,  there is a routing $X\routesto B$ in $D_0$ if and only if there is a routing
  $X\routesto B$ in $D$. Therefore $(D,B,E)$ is a representation of $M$ with the additional property that every $e\in E\BS B$ is a source
  in $D$, and every $b\in B$ is a sink in $D$.
  Thus $(D,B,E)$ is a duality respecting representation of $M$ (Lemma~\ref{lem:dualityrespectingrepresentation}).
\end{proof}

\section{Gammoids with Low Arc-Complexity}

\begin{definition}\PRFR{Mar 27th}
  Let $M$ be a gammoid. The \deftext[arc-complexity of a gammoid]{arc-complexity of $\bm M$}
  is defined to be \label{n:ArcCompl}
  \[ \arcC(M) = \min\SET{\vphantom{A^A} {\left| A \right|} ~\middle|~ \left( (V,A), T, E \right)\text{~is a standard representation of~}M}.\]
\end{definition}

\begin{lemma}\label{lem:arcCkVDualityAndMinors}\PRFR{Mar 27th}
  Let $M=(E,\Ical)$ be a gammoid, $X\subseteq E$. Then the inequalities
   $\arcC(M\restrict X) \leq \arcC(M)$,  $\arcC(M\contract X) \leq \arcC(M)$, and $\arcC(M) = \arcC(M^\ast)$ hold.
\end{lemma}
\begin{proof}\PRFR{Mar 27th}
Let $M$ be a gammoid and let $(D,T,E)$ be a standard representation of $M$ with $D=(V,A)$ for which $\left| A \right|$ is
minimal among all standard representations of $M$. Then $(D^\opp,E\BS T, E)$ is a standard representation of $M^\ast$ 
that uses the same number of arcs. Thus $\arcC(M) = \arcC(M^\ast)$ holds for all gammoids $M$.

Let $X\subseteq E$.
  If $T\subseteq X$, then $(D,T,X)$ is a standard representation of $M\restrict X$ and therefore
  $\arcC(M\restrict X) \leq \arcC(M)$.
  
  If $T\not \subseteq X$ is the case, then let $Y = T\BS X$, 
  and let $B_0\subseteq X$ be a set of maximal cardinality such that
  there is a routing $R_0\colon B_0\routesto Y$ in $D$. 
  Let $D'=(V,A')$ be the digraph that arises from $D$ 
  by the sequence of swaps defined in Remark~\ref{rem:swapsequence} 
  with respect to the routing $R_0$.
  Recall that every $b\in B_0$ is a sink in $D'$ and that $\left| A' \right| = \left| A \right|$.
  We argue that $(D',(T\cap X) \cup B_0, X)$ is a standard representation of $M\restrict X$:
  
  Let $Y_0 = \SET{p_{-1}~\middle|~ p\in R_0}$ be the set of targets that are entered by the routing $R_0$.
  It follows from Corollary~4.1.2 \cite{M72} that the triple
  $(D',(T\cap X) \cup B_0 \cup \left( Y\BS Y_0 \right), E)$ is a representation of $M$. 
  The sequence of operations 
  we carried out on $D$ preserves all those sources and sinks of $D$, which are not visited by a path $p\in R_0$.
  So we obtain that every $e\in E\BS\left( T\cup B_0 \right)$ is a source in $D'$, and that every $t\in T\cap X$
  is a sink in $D'$. Thus the set $T' = (T \cap X)\cup B_0$ consists of sinks in $D'$, and the set
  $X\BS T' \subseteq E\BS\left( T\cup B_0 \right)$ consists of sources in $D'$. Therefore $(D',(T\cap X) \cup B_0, X)$ is a standard
  representation.
  
  We show that $(D',(T\cap X) \cup B_0, X)$ indeed represents $M\restrict X$.
  Clearly, we have that $(D',(T\cap X) \cup B_0 \cup \left( Y\BS Y_0 \right), X)$ is a representation of $M\restrict X$.
  Now, let us assume that $(D',(T\cap X) \cup B_0, X)$ does not represent $M\restrict X$. 
  In this case, there must be a set $X_0\subseteq X$ which is independent in $M\restrict X$, yet dependent in
  $\Gamma(D',(T\cap X) \cup B_0, X)$, because the representation $(D',(T\cap X) \cup B_0, X)$ is basically equal to
  the representation
  $(D',(T\cap X) \cup B_0 \cup \left( Y\BS Y_0 \right), X)$ minus some targets.
  This implies that there is a routing $Q_0\colon X_0\routesto (T\cap X) \cup B_0 \cup \left( Y\BS Y_0 \right)$ in $D'$, 
  but
  there is no routing $X_0\routesto (T\cap X) \cup B_0$ in $D'$. 
  Thus there is a path $q\in Q_0$ with $q_{-1} \in Y\BS Y_0$
  and $q_{1}\in X$. Consequently, we have a routing $Q'_1 = \SET{q}\cup\SET{(b)\in \Pbf(D') ~\middle|~ b\in B_0}$ in $D'$.
  But this implies that there is a routing $B_0\cup\SET{q_{1}} \routesto Y$ in $D$,
  a contradiction to the maximal cardinality of the choice of $B_0$
  above. Thus every independent set of $M\restrict X$ is also independent with respect to $\Gamma(D',(T\cap X) \cup B_0, X)$,
  and $(D',(T\cap X) \cup B_0, X)$ is indeed a standard representation of $M\restrict X$. 
  Consequently
  $\arcC(M\restrict X) \leq \arcC(M)$ holds in the case $T\not \subseteq X$, too.
  
  Finally, we have $\arcC(M\contract X) = \arcC\left( \left( M^\ast \restrict X \right)^\ast \right) =
  \arcC\left( M^\ast \restrict X \right) \leq \arcC(M^\ast) = \arcC(M)$.
\end{proof}

\begin{definition}\PRFR{Mar 27th}
  Let $f\colon \N\maparrow \N\BSET{0}$ be a 
  function. We say that $f$ is \deftext{super-additive},
  if for all $n,m\in \N\BSET{0}$
  \[ f(n+m) \geq f(n) + f(m) \]
  holds.
\end{definition}

\begin{definition}\label{def:arcWf}\PRFR{Mar 27th}
  Let $f\colon \N\maparrow \N\BSET{0}$ be a 
  super-additive function, and let
   $M=(E,\Ical)$ be a gammoid. The \deftext[width of a gammoid]{$\bm f$-width of $\bm M$}
  shall be \label{n:arcWfM}
  \[ \arcW_f(M) = \max\SET{\frac{\arcC\left( \left( M\contract Y \right)\restrict X \right))}{f\left( \left| X \right|
   \right) }
     ~\middle|~ X\subseteq Y\subseteq E}. \qedhere \]
\end{definition}

\begin{theorem}\PRFR{Mar 27th}
  Let $f\colon \N\maparrow \N\BSET{0}$ be a 
  super-additive function, and let $q\in \Qbb$ with $q>0$.
  Let $\Wcal_{f,q}$ denote the class of gammoids $M$ with $\arcW_f(M) \leq q$.
  The class $\Wcal_{f,q}$ is closed under duality, minors, and direct sums.
\end{theorem}

\begin{proof}
  Let $M=(E,\Ical)$ be a gammoid and $X\subseteq Y\subseteq E$.
  It is obvious from Definition~\ref{def:arcWf} that $\arcW_f \left( \left( M\contract Y \right) \restrict X \right) \leq \arcW_f(M)$,
  and consequently $\Wcal_{f,q}$ is closed under minors.
  Since $\arcC(M) = \arcC(M^\ast)$ and since every minor of $M^\ast$ is the dual of a minor of $M$ (\cite{Ox11}, Prop.~3.1.26),
  we obtain that $\arcW_f(M) = \arcW_f(M^\ast)$. Thus $\Wcal_{f,q}$ is closed under duality.

     Now, let $M=(E,\Ical)$ and $N=(E',\Ical')$ with $E\cap E' =\emptyset$ and $M,N\in \Wcal_{f,q}$.
     The cases where either $E = \emptyset$ or $E' = \emptyset$ are trivial, now let $E\not= \emptyset \not= E'$.
   Furthermore, let $X\subseteq Y\subseteq E\cup E'$.
    The direct sum commutes with the forming of minors in the sense that
   \begin{align*}
     \left( (M\oplus N)\contract Y \right)\restrict X &  \,\,\,=\,\,\,
      \left(\vphantom{A^A}\left( M\contract Y\cap E \right)\restrict X\cap E\right)  \oplus 
      \left( ( N\contract Y\cap E' )\restrict X\cap E'\right).
   \end{align*}
   Let $(D,T,E)$ and $(D',T',E')$ be representations of $M$ and $N$ where $D=(V,A)$ and $D'=(V',A')$ such that $V\cap V' = \emptyset$.
   Then $\left(\vphantom{A^A} (V\cup V', A\cup A'), T\cup T', E\cup E' \right)$ is a representation of $M\oplus N$,
   and consequently $\arcC(M\oplus N) \leq \arcC(M) + \arcC(N)$ holds for all gammoids $M$ and $N$,
   thus we have
   \begin{align*}
    \arcC\left(\vphantom{A^A} \left( (M\oplus N)\contract Y \right)\restrict X  \right) & \,\,\,\leq \,\,\,
    \arcC\left(\vphantom{A^A} M_{X,Y}\right) 
    + \arcC\left(\vphantom{A^A} N_{X,Y} \right) 
   \end{align*}
   where $M_{X,Y} = \left( M\contract Y\cap E \right)\restrict X\cap E$ and $N_{X,Y} = ( N\contract Y\cap E' )\restrict X\cap E'$.
   The cases where $\arcC(M_{X,Y}) = 0$  or $\arcC(N_{X,Y}) = 0$ are trivial, so we may assume that $X\cap E \not= \emptyset \not= X\cap E'$.
   We use the super-additivity of $f$ at $(\ast)$ in order to derive
   \begin{align*}
    \frac {\arcC\left(\vphantom{A^A} \left( (M\oplus N)\contract Y \right)\restrict X  \right)}{f\left( \left| X \right|
    \right)} &
    \,\,\, \leq\,\,\,
    \frac { \arcC\left(\vphantom{A^A} M_{X,Y}\right) 
    + \arcC\left(\vphantom{A^A} N_{X,Y}\right) }{f\left( \left| X \right|
    \right) } \\
    &
    \,\,\, \leq \,\,\, \frac{q\cdot f\left( \left| X\cap E \right|
    \right) + q\cdot f\left( \left| X\cap E' \right|
     \right)}{f\left( \left| X \right|
     \right) }  \,\,\, \stackrel{(\ast)}{\leq} \,\,\, q .
   \end{align*}
   As a consequence we obtain $\arcW_f(M\oplus N) \leq q$, and therefore
   $\Wcal_{f,q}$ is closed under direct sums.
\end{proof}

\section{Further Remarks and Open Problems}

A consequence of a result of S.~Kratsch and M.~Wahlström (\cite{KW12}, Thm.~3) is, that
if a matroid $M=(E,\Ical)$ is a gammoid, then there is a representation $(D,T,E)$ of $M$ with $D=(V,A)$
and  $\left| V \right|  \leq \rk_M(E)^2 \cdot \left| E \right| + \rk_M(E) + \left| E \right| $.
It is easy to see that if $M \in \Wcal_{f,q}$, then there is a representation $(D,T,E)$ of $M$ with $D=(V,A)$
and $\left| V \right| \leq  \left\lfloor 2 q\cdot f\left( \left| E \right| \right) \right\rfloor$,
since every arc is only incident with at most two vertices.
Therefore, deciding $\Wcal_{f,q}$-membership with an exhaustive digraph search appears to be
easier than deciding gammoid-membership with an exhaustive digraph search.

Furthermore, the above result yields the following: Take a representation $(D,T,E)$ with $D=(V,A)$ and 
$\left| V \right|  \leq \rk_M(E)^2 \cdot \left| E \right| + \rk_M(E) + \left| E \right| $. Clearly, 
any arc $(v,v)\in A$ may be dropped 
without affecting the represented gammoid, so we assume $(v,v) \notin A$ for all $v\in V$.
 Using the construction
from the proof of Theorem~\ref{thm:two-four} yields a representation $(D_1,T,E)$ with $D_1=(V_1,A_1)$ and
$ \left| V_1 \right| = \left| V \right| + \left| E \right| $.  Furthermore, $\left| A_1 \right| = \left|A\right| + \left| E \right|\leq \left| V \right| \cdot \left|V \right| - \left| V \right| + \left| E \right|
\leq \left| V \right|^2$. Thus, there is an immediate upper bound 
\begin{flalign*}\arcC(M) \leq\,\,\,\,\,\,\,\, & 
\rk_M(E)^4 \cdot \left| E \right|^2 + 2\cdot\rk_M(E)^3 \cdot \left| E \right|
+ 2\cdot\rk_M(E)^2 \cdot \left| E \right|^2 
\\& + \rk_M(E)^2 + 2\cdot\rk_M(E)\cdot \left| E \right|
+ \left| E \right|^2  .
\end{flalign*}

Let $r,n\in \N$ with $n\geq r$, the uniform matroid of rank $r$ on $n$ elements is the matroid
\( U_{r,n} = \left( \SET{1,2,\ldots,n}, \Ical_{r,n} \right) \) where $ \Ical_{r,n} = \SET{\vphantom{A^A}X\subseteq \SET{1,2,\ldots,n}~\middle|~ \left| X \right|\leq r}$.
Let $T=\SET{1,2,\ldots,r}$, $X = \SET{r+1,r+2,\ldots,n}$, and $D=(X\cup T, X\times T)$. Then
$U_{r,n} = \Gamma(D,T,T\cup X)$. Thus $\arcC\left( U_{r,n} \right) \leq r \cdot \left( n-r \right)$.
Unfortunately, we were not able to find a known result in graph or digraph theory that implies:
\begin{conjecture}
  \[\arcC\left( U_{r,n} \right) = r\cdot(n-r).\]
\end{conjecture}
Let $\Fbb$ be a finite field. Is there an upper bound $c_\Fbb \in \Nbb$ such that we have $\arcC(M) \leq c_\Fbb\cdot \left| E\right|$ for every 
 gammoid $M=(E,\Ical)$ that is representable over $\Fbb$?

\begin{conjecture}\label{conj:complexGammoids}
  For every $q\in \Qbb$ there is a gammoid $M=(E,\Ical)$ with $$\arcC(M) \geq q\cdot \left| E \right|.$$
\end{conjecture}

Is the set 
 $\SET{\frac{\arcC(M)}{\left| E \right|} ~\middle|~ M=\left(E,\Ical \right) \text{gammoid}}$ 
 dense in the positive reals?\footnote{This question has been suggested by a reviewer.}

Let us fix some super-additive \( \hat{f}\colon \N \maparrow \N\BSET{0},\,x\mapsto \max\SET{1,x} \).
 Clearly, if Conjecture~\ref{conj:complexGammoids} holds, then $(\Wcal_{\hat{f},i})_{i\in \N\BSET{0}}$ contains an infinite, strictly monotonous subsequence
 of subclasses of the class of gammoids, such that every subclass in that subsequence is closed under duality, minors, and direct sums.

 For which super-additive $f$ and $q\in \Qbb\BSET{0}$ may $\Wcal_{f,q}$ be characterized by finitely many excluded minors?
 
 For which such classes can we list a sufficient (possibly infinite) set of excluded minors that decide class membership of $\Wcal_{f,q}$?


\bigskip
\footnotesize
\noindent\textit{Acknowledgments.}
This research was partly supported by a scholarship granted by the FernUniversität in Hagen.


\section*{References}

\bibliography{drr}

\end{document}